\documentclass[12pt]{article}
\usepackage{graphicx}
\usepackage{appendix}
\usepackage{amsmath}
\usepackage{amssymb}
\usepackage{amsthm}
\usepackage{multirow}
\usepackage{color}
\usepackage{longtable}
\usepackage{array}
\usepackage{url}
\usepackage{booktabs}
\usepackage{float}
\usepackage{mathtools}
\usepackage{tikz}
\usepackage{caption}
\usepackage{subfigure}

\newtheorem{theorem}{Theorem}[section]
\newtheorem{definition}[theorem]{Definition}
\newtheorem{lemma}[theorem]{Lemma}

\newtheorem{corollary}[theorem]{Corollary}

\newtheorem{conjecture}[theorem]{Conjecture}

\theoremstyle{definition}
\newtheorem{remark}[theorem]{Remark}

\title{A group ring approach to Fuglede's conjecture in cyclic groups}
\author{Tao Zhang\thanks{Email address:  zhant220@163.com.}\\
\footnotesize  Zhejiang Lab, Hangzhou 311100, China. \\
}

\begin{document}
\date{}

\maketitle

\begin{abstract}
Fuglede's conjecture states that a subset $\Omega\subseteq\mathbb{R}^{n}$ of positive and finite Lebesgue measure is a spectral set if and only if it tiles $\mathbb{R}^{n}$ by translation.  The conjecture does not hold in both directions for $\mathbb{R}^n$, $n\ge3$. However, this conjecture remains open in $\mathbb{R}$ and $\mathbb{R}^2$. Cyclic groups play important roles in the study of Fuglede's conjecture in $\mathbb{R}$. In this paper, we introduce a new tool to study the spectral sets in cyclic groups. In particular, we prove that Fuglede's conjecture holds in $\mathbb{Z}_{p^{n}qr}$.

\medskip


\smallskip

\end{abstract}

\section{Introduction}
A bounded measurable subset $\Omega\subseteq\mathbb{R}^{n}$ with $\mu(\Omega)>0$ is called spectral, if there is a subset $\Lambda\subseteq\mathbb{R}^{n}$ such that the set of exponential functions $\{e_{\lambda}(x)\}_{\lambda\in\Lambda}$ is a complete orthogonal basis, where $e_{\lambda}(x)=e^{2\pi i\langle\lambda,x\rangle}$. In this case, $\Lambda$ is called the spectrum of $\Omega$, and $(\Omega,\Lambda)$ is called a spectral pair in $\mathbb{R}^n$.

A subset $A\subseteq\mathbb{R}^{n}$ tiles $\mathbb{R}^{n}$ by translation, if there is a set $T\subseteq\mathbb{R}^{n}$ such that almost all elements of $\mathbb{R}^{n}$ can be uniquely written as a sum $a+t$, where $a\in A$, $t\in T$. We will denote this by $A\oplus T=\mathbb{R}^{n}$. $T$ is called the tiling complement of $A$, and $(A,T)$ is called a tiling pair in $\mathbb{R}^{n}$.

In 1974, Fuglede \cite{F74} proposed the following conjecture, which connected these two notions.
\begin{conjecture}
A subset $\Omega\subseteq\mathbb{R}^{n}$ of positive and finite Lebesgue measure is a spectral set if and only if it tiles $\mathbb{R}^{n}$ by translation.
\end{conjecture}
In the same paper, Fuglede proved this conjecture when the tiling complement or the spectrum is a lattice in $\mathbb{R}^n$. 30 years later, Tao \cite{T04} disproved this conjecture by constructing a non-tile spectral set in $\mathbb{R}^5$. Currently, the conjecture does not hold in both directions for $\mathbb{R}^n$, $n\ge3$ \cite{FMM06,KM2006,KM06,M05}. However, this conjecture remains open in $\mathbb{R}$ and $\mathbb{R}^2$.

Since original Fuglede's conjecture falses for $\mathbb{R}^n$, $n\ge3$, then researchers considered this problem from two persepctives. One is under additional assumptions. In 2003, Iosevich, Katz and Tao \cite{IKT03} showed that Fuglede's conjecture holds for convex sets in $\mathbb{R}^{2}$. Later, a similar result in dimension 3 was proved by Greenfeld and Lev \cite{GL17}. Recently, Lev and Matolcsi \cite{LM21} proved that Fuglede's conjecture holds for convex domains in $\mathbb{R}^{n}$ for all $n$. Another is trying to find for which group $G$, Fuglede's conjecture holds in $G$. In \cite{FFLS19,FFS16}, Fan et al. proved that Fuglede's conjecture holds in $\mathbb{Q}_{p}$, the field of $p$-adic numbers. We also know that Fuglede's conjecture holds in the following finite Abelian groups: $\mathbb{Z}_{p}^{d}$ ($p=2$ and $d\le6$; $p$ is an odd prime and $d=2$; $p=3,5,7$ and $d=3$) \cite{AABF17,FMV2022,FS20,IMP17}, $\mathbb{Z}_{p}\times\mathbb{Z}_{p^{n}}$ \cite{IMP17,S20,Z2022}, $\mathbb{Z}_{p}\times\mathbb{Z}_{pq}$ \cite{KS2021}  and $\mathbb{Z}_{pq}\times\mathbb{Z}_{pq}$ \cite{FKS2012}, $\mathbb{Z}_{p^{n}}$ \cite{L02}, $\mathbb{Z}_{p^{n}q^{m}}$ ($p<q$ and $m\le9$ or $n\le6$; $p^{m-2}<q^{4}$) \cite{KMSV20,M21,MK17}, $\mathbb{Z}_{pqr}$ \cite{S19}, $\mathbb{Z}_{p^{2}qr}$ \cite{Somlai21} and $\mathbb{Z}_{pqrs}$ \cite{KMSV2012}, where $p,q,r,s$ are distinct primes.

In this paper, we focus on finite cyclic groups. Following the notations from \cite{DL2014}, write $S-T(G)$ (respectively, $T-S(G)$), if the ``Spectral $\Rightarrow$ Tile'' (respectively, ``Tile $\Rightarrow$ Spectral'') direction of Fuglede's conjecture holds in $G$. Then we have the following relations \cite{DL2013,DL2014}:
\[T-S(\mathbb{R})\Leftrightarrow T-S(\mathbb{Z})\Leftrightarrow T-S(\mathbb{Z}_{N})\text{ for all }N,\]
and
\[S-T(\mathbb{R})\Rightarrow S-T(\mathbb{Z})\Rightarrow S-T(\mathbb{Z}_{N})\text{ for all }N.\]
The above relations show that finite cyclic groups play important roles in the study of Fuglede's conjecture in $\mathbb{R}$.
As we have seen, Fuglede's conjecture holds in the following finite cyclic groups: $\mathbb{Z}_{p^{n}}$ \cite{L02}, $\mathbb{Z}_{p^{n}q^{m}}$ ($p<q$ and $m\le9$ or $n\le6$; $p^{m-2}<q^{4}$) \cite{KMSV20,M21,MK17}, $\mathbb{Z}_{pqr}$ \cite{S19}, $\mathbb{Z}_{p^{2}qr}$ \cite{Somlai21} and $\mathbb{Z}_{pqrs}$ \cite{KMSV2012}, where $p,q,r,s$ are distinct primes. For the direction ``Tile $\Rightarrow$ Spectral'', \L aba \cite{L02} proved $T-S(\mathbb{Z}_{p^nq^m})$ for distinct primes $p,q$. Later, \L aba and Meyerowitz proved $T-S(\mathbb{Z}_{n})$ in comments of Tao's blog \cite{T2011} (see also \cite{S19}), where $n$ is a squarefree integer. Recently, Malikiosis \cite{M21} proved $T-S(\mathbb{Z}_{p_{1}^{n}p_2\cdots p_k})$, where $p_1,p_2,\dots,p_k$ are distinct primes. In \cite{LL22,LL2022}, the authors developed some new tools to study tiling sets in cyclic groups and proved $T-S(\mathbb{Z}_{p^2q^2r^2})$, where $p,q,r$ are distinct primes.

Now we state our main result.
\begin{theorem}\label{mainthm}
Let $p,q,r$ be distinct primes and $n$ be a positive integer. A subset in $\mathbb{Z}_{p^{n}qr}$ is a spectral set if and only if it is a tile of $\mathbb{Z}_{p^{n}qr}$.
\end{theorem}

Note that the ``Tile $\Rightarrow$ Spectral'' direction follows from \cite{M21}. Hence, we only need to prove the ``Spectral $\Rightarrow$ Tile'' direction.
When we consider Fuglede's conjecture in cyclic groups, one of the most important tools is the so-called (T1) and (T2) conditions, which was introduced by Coven and Meyerowitz \cite{CM99}. In this paper, we introduce a new tool ``group ring'' to study spectral sets in cyclic groups. In particular, we prove that Fuglede's conjecture holds in $\mathbb{Z}_{p^{n}qr}$. This paper is organized as follows. In Section~\ref{section2}, we recall some basics of spectral sets and tiles in cyclic groups. In Section~\ref{section3}, we prove some useful lemmas by group ring. In Section~\ref{section4}, we prove Theorem~\ref{mainthm}.

\section{Preliminaries}\label{section2}
Let $\mathbb{Z}_{n}$ be a finite cyclic group with order $n$ (written additively).
For any $a\in\mathbb{Z}_{n}$, define
\[\chi_{a}(b)=e^{\frac{2\pi i\cdot ab}{n}},\]
and $\chi_{a}\chi_{b}=\chi_{a+b}$. Then the set $\widehat{\mathbb{Z}_{n}}=\{\chi_{a}:\ a\in\mathbb{Z}_{n}\}$ forms a group which is isomorphic to $\mathbb{Z}_{n}$.

Now we restate the definition of spectral sets and tiles in cyclic groups.
\begin{definition}
A subset $A\subseteq \mathbb{Z}_{N}$ is said to be spectral if there is a subset $B\subseteq \mathbb{Z}_{N}$ such that
\[\{\chi_{b}: b\in B\}\]
forms an orthogonal basis in $L^{2}(A)$, the vector space of complex valued functions on $A$ with Hermitian inner product $\langle f,g\rangle=\sum_{a\in A}f(a)\overline{g}(a)$. In such a case, the set $B$ is called a spectrum of $A$, and $(A,B)$ is called a spectral pair.
\end{definition}
Since the dimension of $L^{2}(A)$ is $|A|$, the pair $(A,B)$ being a spectral pair is equivalent to
\[|A|=|B|\textup{ and }\sum_{a\in A}\chi_{b-b'}(a)=0\textup{ for all }b\ne b'\in B.\]
The set of zeros of $A$ is defined by
\[\mathcal{Z}_{A}=\{b\in \mathbb{Z}_{n}: \sum_{a\in A}\chi_{b}(a)=0\}.\]

The following equivalent conditions of a spectral pair can be found in \cite{S20,Z2022}.
\begin{lemma}\label{lem3p1}
Let $A,B\subseteq \mathbb{Z}_{N}$. Then the following statements are equivalent.
\begin{enumerate}
  \item[(a)] $(A,B)$ is a spectral pair.
  \item[(b)] $(B,A)$ is a spectral pair.
  \item[(c)] $|A|=|B|$ and $(B-B)\backslash\{0\}\subseteq\mathcal{Z}_{A}$.
  \item[(d)] The pair $(aA+g,bB+h)$ is a spectral pair for all $a,b\in\mathbb{Z}_{N}^{*}$ and $g,h\in \mathbb{Z}_{N}$.
\end{enumerate}
\end{lemma}

\begin{definition}
A subset $A\subseteq \mathbb{Z}_{N}$ is said to be a tile if there is a subset $T\subseteq \mathbb{Z}_{N}$ such that each element $g\in \mathbb{Z}_{N}$ can be expressed uniquely in the form
\[g=a+t,\ a\in A,\ t\in T.\]
We will denote this by $\mathbb{Z}_{N}=A\oplus T$.
The set $T$ is called a tiling complement of $A$, and $(A,T)$ is called a tiling pair.
\end{definition}

We have the following equivalent conditions for a tiling pair \cite{S20}, \cite[Lemma 2.1]{SS09}.
\begin{lemma}\label{lem3p2}
Let $A,T$ be subsets in $\mathbb{Z}_{N}$. Then the following statements are equivalent.
\begin{enumerate}
  \item[(a)] $(A,T)$ is a tiling pair.
  \item[(b)] $(T,A)$ is a tiling pair.
  \item[(c)] $(A+g,T+h)$ is a tiling pair.
  \item[(d)] $|A|\cdot|T|=N$ and $(A-A)\cap(T-T)=\{0\}$.
  \item[(e)] $|A|\cdot|T|=N$ and $\mathcal{Z}_{A}\cup\mathcal{Z}_{T}=\mathbb{Z}_{N}\backslash\{0\}$.
\end{enumerate}
\end{lemma}

If $|A|=1$ or $A=\mathbb{Z}_{N}$, then the set $A$ is called trivial. It is easy to see that a trivial set is a spectral set and also a tiling set. In the following of this paper, we will only consider nontrivial sets.
We also need the following lemmas, which will be useful in the following sections.
\begin{lemma}\label{lem3p3}\cite{KMSV20}
Let $A$ be a spectral set in $\mathbb{Z}_{N}$, that does not generate $\mathbb{Z}_{N}$. Assume that for every proper subgroup $H$ of $\mathbb{Z}_{N}$ we have $S-T(H)$. Then $A$ tiles $\mathbb{Z}_{N}$.
\end{lemma}

\begin{lemma}\label{lem3p4}\cite{KMSV20}
Let $N$ be a natural number and suppose that $S-T(\mathbb{Z}_{N}/H)$ holds for every $\{0\}\ne H\leq \mathbb{Z}_{N}$. Assume that $(A,B)$ is a spectral pair and $B$ does not generate $\mathbb{Z}_{N}$. Then $A$ tiles $\mathbb{Z}_{N}$.
\end{lemma}

\begin{lemma}\label{lem3p5}\cite{KMSV20}
Let $N$ be a natural number, $A$ a spectral set in $\mathbb{Z}_{N}$ and $p$ a prime divisor of $N$. Assume that $S-T(\mathbb{Z}_{\frac{N}{p}})$. If $A$ is the union of $\mathbb{Z}_{p}$-cosets, then $A$ tiles $\mathbb{Z}_{N}$.
\end{lemma}

\begin{lemma}\label{lemma2}\cite{Somlai21}
Let $0\in T\subseteq\mathbb{Z}_{N}$ be a generating set and assume that $p$ and $q$ are different prime divisors of $N$. Then there are elements $t_{1}\ne t_{2}\in T$ such that $p\nmid(t_{1}-t_{2})$ and $q\nmid(t_{1}-t_{2})$.
\end{lemma}

\begin{lemma}\label{lemma1}
  Let $p$ be a prime and set $\zeta=\zeta_{p^n}$, a primitive $p^n$-th root of unity. Let $c=c_{p^n-1}\zeta^{p^n-1}+c_{p^n-2}\zeta^{p^n-2}+\cdots+c_1\zeta+c_0$, where $c_i\in\mathbb{Z}$, $0\le i\le p^n-1$. Then $c=0$ if and only if $c_i=c_j$ for any $i,j$ with $i\equiv j\pmod{p^{n-1}}$.
\end{lemma}
\begin{proof}
  Let $f(x)=c_{p^n-1}x^{p^n-1}+c_{p^n-2}x^{p^n-2}+\cdots+c_1x+c_0$, then $c=0$ if and only if $\zeta$ is a root of $f(x)$. Since the minimal polynomial of $\zeta$ over $\mathbb{Z}$ is
  \[\Phi_{p^{n}}(x)=x^{(p-1)p^{n-1}}+x^{(p-2)p^{n-1}}+\cdots+x^{p^{n-1}}+1,\]
  then $c=0$ if and only if there exists a polynomial $g(x)\in\mathbb{Z}[x]$ such that
  \[f(x)=\Phi_{p^{n}}(x)g(x).\]
  Hence, the statement follows.
\end{proof}

\begin{lemma}\label{lemma7}
Let $V\subset\mathbb{Z}_{p^n}$ with $|V|=p^{t}$, $t\le n$. Let $I\subseteq[0,n-1]$, $|I|=t$ and $n-1\in I$. If $0\in V$ and $V-V\subseteq\{\sum_{i\in I}a_ip^i: a_i\in[0,p-1]\}$, then $V=\{\sum_{i\in I}a_ip^i: a_i\in[0,p-1]\}$.
\end{lemma}
\begin{proof}
  We prove the lemma by induction. If $|I|=1$, then $I=\{n-1\}$. It is easy to see that $V=\{ap^{n-1}:\ a\in[0,p-1]\}$. Suppose that the statement holds for $|I|< t$.

  Let $|I|=t$, $I=\{i_j:\ j\in[1,t]\}$, and $0\le i_1<i_2<\cdots<i_t= n-1$. For any $v\in V$, we can write $v$ as $v=\sum_{i=0}^{n-1}v_ip^i$, where $v_i\in[0,p-1]$.
    Since $0\in V$ and $V-V\subseteq\{\sum_{i\in I}a_ip^i: a_i\in[0,p-1]\}$, we have $v_i=0$ for $i<i_1$. Denote
  \[V_k=\{v\in V:\ v_{i_1}=k\}.\]
  Then $V=\cup_{k=0}^{p-1}V_k$.
By the pigeonhole principle, there exists $k$ such that $|V_k|\ge p^{t-1}$. Note that
\[V_k-V_k\subseteq\{\sum_{i\in I\backslash\{i_1\}}a_ip^i: a_i\in[0,p-1]\}.\]
By the pigeonhole principle again, we have $|V_{k}|\le p^{t-1}$. Hence $|V_{k}|= p^{t-1}$ for all $k\in[0,p-1]$. By induction, we have $V_{k}=\{kp^{i_1}+\sum_{i\in I\backslash\{i_1\}}a_ip^i: a_i\in[0,p-1]\}$, and so $V=\{\sum_{i\in I}a_ip^i: a_i\in[0,p-1]\}$.
\end{proof}

\section{Technique tools}\label{section3}
Throughout the following sections, cyclic group $\mathbb{Z}_{N}$ will be written multiplicatively. Let $\mathbb{Z}_{N}=\langle u\rangle$, then all the statements in Section~\ref{section2} still hold under the isomorphism map: $i\rightarrow u^{i}$.

Let $\mathbb{Z}[\mathbb{Z}_{N}]$ denote the group ring of $\mathbb{Z}_{N}$ over $\mathbb{Z}$. For any $X\in\mathbb{Z}[\mathbb{Z}_{N}]$, $X$ can be written as formal sums $X=\sum_{g\in\mathbb{Z}_{N}}x_{g}g$, where $x_{g}\in\mathbb{Z}$. The addition and subtraction of elements in $\mathbb{Z}[\mathbb{Z}_{N}]$ is defined componentwise, i.e.
\[\sum_{g\in\mathbb{Z}_{N}}x_{g}g\pm\sum_{g\in\mathbb{Z}_{N}}y_{g}g:=\sum_{g\in\mathbb{Z}_{N}}(x_{g}\pm y_g)g.\]
The multiplication is defined by
\[(\sum_{g\in\mathbb{Z}_{N}}x_{g}g)(\sum_{g\in\mathbb{Z}_{N}}y_{g}g):=\sum_{g\in\mathbb{Z}_{N}}(\sum_{h\in\mathbb{Z}_{N}}x_hy_{h^{-1}g})g.\]
For $X=\sum_{g\in\mathbb{Z}_{N}}x_{g}g$ and $t\in\mathbb{Z}$, we define
\[X^{(t)}:=\sum_{g\in\mathbb{Z}_{N}}x_{g}g^{t}.\]
For any set $X$ whose elements belong to $\mathbb{Z}_{N}$ ($X$ may be a multiset), we can identify $X$ with the group ring element $\sum_{g\in\mathbb{Z}_{N}}x_{g}g$, where $x_g$ is the multiplicity of $g$ appearing in $X$.

For any $g=u^a,h=u^b\in\mathbb{Z}_{N}$, define
\[\chi_{g,N}(h):=e^{\frac{2\pi i\cdot ab}{N}}.\]
We will use $\chi_{a,N}$ instead of $\chi_{u^a,N}=\chi_{g,N}$ if there is no misunderstanding. For any $\chi \in \widehat{\mathbb{Z}_{N}}$ and $X=\sum_{g\in\mathbb{Z}_{N}}x_{g}g\in\mathbb{Z}[\mathbb{Z}_{N}]$, define
\[\chi(X):=\sum_{g\in \mathbb{Z}_{N}} x_{g} \chi(g).\]
Then the pair $(A,B)$ forms a spectral pair if and only if
\[|A|=|B|\textup{ and }\chi_{b-b',N}(A)=0\textup{ for all }u^b\ne u^{b'}\in B.\]

Let $\mathbb{Z}_{p^{n}p_{1}\cdots p_{k}}=\langle a,a_1,\dots,a_{k}\rangle$, where $o(a)=p^{n}$, $o(a_{i})=p_i$ for $i=1,\dots,k$. Let $A$ be a subset of $\mathbb{Z}_{p^{n}p_{1}\cdots p_{k}}$, then $A$ can be written as $A=\sum_{i_1=0}^{p_1-1}\cdots\sum_{i_{k}=0}^{p_k-1}A_{i_1\dots i_k}a_{1}^{i_1}\cdots a_{k}^{i_k}$, where $A_{i_1\dots i_k}\in \mathbb{Z}_{\ge0}[\langle a\rangle]$. Denote
\[\mathcal{I}_{t,s}:=\{(i_1,i_2,\dots,i_k):\text{ there are exactly } s \text{ of } j\in[t+1,k] \text{ such that } i_{j}=0\}.\]
Let $A_{\mathcal{I}_{t,s}}:=\sum_{I\in \mathcal{I}_{t,s}}A_{I}.$
Then we have the following lemma, which can transfer the problem from $\mathbb{Z}_{p^{n}p_{1}\cdots p_{k}}$ to $\mathbb{Z}_{p^{n}}$.
\begin{lemma}
Let $0\le t\le k$, $0\le i\le n$, then $p^{i}p_{1}\cdots p_{t}\in\mathcal{Z}_{A}$ if and only if \[\chi_{p^{i},p^{n}}(\sum_{s=0}^{k-t}\sum_{j=1}^{t}\sum_{i_j=0}^{p_j-1}(-1)^{s}A_{\mathcal{I}_{t,s}})=0\]
 for all $l\in[t+1,k]$, $i_l\in[0,p_j-1]$, where $p^{i}p_{1}\cdots p_{t}:=p^{i}$ if $t=0$.
\end{lemma}
\begin{proof}
  By the definition of zeros of a set, we have $p^{i}p_{1}\cdots p_{t}\in\mathcal{Z}_{A}$ if and only if
  \begin{align*}
  0&=\chi_{p^{i}p_{1}\cdots p_{t},p^{n}p_{1}\cdots p_{k}}(A)\\
   &=\chi_{p^{i}p_{1}\cdots p_{t},p^{n}p_{1}\cdots p_{k}}(\sum_{i_1=0}^{p_1-1}\cdots\sum_{i_{k}=0}^{p_k-1}A_{i_1\dots i_k}a_{1}^{i_1}\cdots a_{k}^{i_k})\\
   &=\sum_{i_1=0}^{p_1-1}\cdots\sum_{i_{k}=0}^{p_k-1}\chi_{p^{i},p^n}(A_{i_1\dots i_k})\zeta_{p_{t+1}}^{i_{t+1}}\cdots \zeta_{p_{k}}^{i_k}\\
   &=\sum_{i_{k}=0}^{p_k-1}(\sum_{i_{t+1}=0}^{p_{t+1}-1}\cdots\sum_{i_{k-1}=0}^{p_{k-1}-1}\chi_{p^{i},p^n}(\sum_{i_1=0}^{p_1-1}\cdots\sum_{i_{t}=0}^{p_t-1}A_{i_1\dots i_k})\zeta_{p_{t+1}}^{i_{t+1}}\cdots \zeta_{p_{k-1}}^{i_{k-1}})\zeta_{p_{k}}^{i_k}\\
   &=\sum_{i_{k}=1}^{p_k-1}(\sum_{i_{t+1}=0}^{p_{t+1}-1}\cdots\sum_{i_{k-1}=0}^{p_{k-1}-1}\chi_{p^{i},p^n}(\sum_{i_1=0}^{p_1-1}\cdots\sum_{i_{t}=0}^{p_t-1}(A_{i_1\dots i_k}-A_{i_1\dots i_{k-1}0}))\zeta_{p_{t+1}}^{i_{t+1}}\cdots \zeta_{p_{k-1}}^{i_{k-1}})\zeta_{p_{k}}^{i_k}
  \end{align*}
  Since $\zeta_{p_k},\zeta_{p_k}^{2},\dots,\zeta_{p_k}^{p_k-1}$ forms a basis of $\mathbb{Q}(\zeta_{p^{n}p_{1}\cdots p_{k}})/\mathbb{Q}(\zeta_{p^{n}p_{1}\cdots p_{k-1}})$, then $\chi_{p^{i}p_{1}\cdots p_{t},p^{n}p_{1}\cdots p_{k}}(A)=0$ is equivalent to $\sum_{i_{t+1}=0}^{p_{t+1}-1}\cdots\sum_{i_{k-1}=0}^{p_{k-1}-1}\chi_{p^{i},p^n}(\sum_{i_1=0}^{p_1-1}\cdots\sum_{i_{t}=0}^{p_t-1}(A_{i_1\dots i_k}-A_{i_1\dots i_{k-1}0}))\zeta_{p_{t+1}}^{i_{t+1}}\cdots \zeta_{p_{k-1}}^{i_{k-1}}=0$ for all $i_k\in[0,p_k-1]$. Repeating above arguments, we have the statement.
\end{proof}

In particular, let $\mathbb{Z}_{p^{n}qr}=\langle a,b,c\rangle$, where $o(a)=p^{n}$, $o(b)=q$ and $o(c)=r$, and write $A=\sum_{j=0}^{q-1}\sum_{k=0}^{r-1}A_{jk}b^{j}c^{k}$, where $A_{jk}\in\mathbb{Z}_{\ge0}[\langle a\rangle]$. Then we have the following corollary.
\begin{corollary}\label{coro1}
\begin{enumerate}
  \item[(1)] $p^{i}\in\mathcal{Z}_{A}$ if and only if $\chi_{p^{i},p^n}(A_{jk}-A_{j0}-A_{0k}+A_{00})=0$ for all $j\in[0,q-1],k\in[0,r-1]$.
  \item[(2)] $p^{i}q\in\mathcal{Z}_{A}$ if and only if $\chi_{p^{i},p^n}(\sum_{j=0}^{q-1}(A_{jk}-A_{j0}))=0$ for all $k\in[0,r-1]$.
  \item[(3)] $p^{i}r\in\mathcal{Z}_{A}$ if and only if $\chi_{p^{i},p^n}(\sum_{k=0}^{r-1}(A_{jk}-A_{0k}))=0$ for all $j\in[0,q-1]$.
  \item[(4)] $p^{i}qr\in\mathcal{Z}_{A}$ if and only if $\chi_{p^{i},p^n}(\sum_{j=0}^{q-1}\sum_{k=0}^{r-1}A_{jk})=0$.
\end{enumerate}
\end{corollary}

If $A$ has many zeros, then we can get more information about the sets $A_{jk}$, $j\in[0,q-1],k\in[0,r-1]$.
\begin{lemma}\label{lemma3}
  \begin{enumerate}
    \item [(1)] If $p^{i},p^{i}q\in\mathcal{Z}_{A}$, then $\chi_{p^{i},p^n}(A_{jk}-A_{j0})=0$ for all $j\in[0,q-1],k\in[0,r-1]$.
    \item [(2)] If $p^{i},p^{i}r\in\mathcal{Z}_{A}$, then $\chi_{p^{i},p^n}(A_{jk}-A_{0k})=0$ for all $j\in[0,q-1],k\in[0,r-1]$.
    \item [(3)] If $p^{i}q,p^{i}r\in\mathcal{Z}_{A}$, then $r\chi_{p^i,p^n}(\sum_{j=0}^{q-1}A_{jk})=q\chi_{p^i,p^n}(\sum_{k=0}^{r-1}A_{jk})$ for all $j\in[0,q-1],k\in[0,r-1]$.
    \item [(4)] If $p^{i}q,p^{i}qr\in\mathcal{Z}_{A}$, then $\chi_{p^{i},p^n}(\sum_{j=0}^{q-1}A_{jk})=0$ for all $k\in[0,r-1]$.
    \item [(5)] If $p^{i}r,p^{i}qr\in\mathcal{Z}_{A}$, then $\chi_{p^{i},p^n}(\sum_{k=0}^{r-1}A_{jk})=0$ for all $k\in[0,r-1]$.
    \item [(6)] If $p^{i},p^{i}q,p^{i}r\in\mathcal{Z}_{A}$, then $\chi_{p^{i},p^n}(A_{jk}-A_{00})=0$ for all $j\in[0,q-1],k\in[0,r-1]$.
    \item [(7)] If $p^{i},p^{i}r,p^{i}q,p^{i}qr\in\mathcal{Z}_{A}$, then $\chi_{p^{i},p^n}(A_{jk})=0$ for all $j\in[0,q-1],k\in[0,r-1]$.
  \end{enumerate}
\end{lemma}
\begin{proof}
  We will only prove (1) and (3). For other statements, the proofs are similar.

  (1). If $p^{i},p^{i}q\in\mathcal{Z}_{A}$, by Corollary~\ref{coro1}, we have
  \begin{align*}
    & \chi_{p^{i},p^n}(A_{jk}-A_{j0}-A_{0k}+A_{00})=0, \\
    & \chi_{p^{i},p^n}(\sum_{j=0}^{q-1}(A_{jk}-A_{j0}))=0.
  \end{align*}
  Then we can compute to get that
  \begin{align*}
  0&=\sum_{j=0}^{q-1}\chi_{p^{i},p^n}(A_{jk}-A_{j0}-A_{0k}+A_{00})\\
   &=\chi_{p^{i},p^n}(\sum_{j=0}^{q-1}(A_{jk}-A_{j0}-A_{0k}+A_{00}))\\
   &=\chi_{p^{i},p^n}(\sum_{j=0}^{q-1}(-A_{0k}+A_{00}))\\
   &=q\chi_{p^{i},p^n}(-A_{0k}+A_{00})\\
   &=q\chi_{p^{i},p^n}(-A_{jk}+A_{j0}).
  \end{align*}
  Hence $\chi_{p^{i},p^n}(A_{jk}-A_{j0})=0$ for all $j\in[0,q-1],k\in[0,r-1]$.

  (3) If $p^{i}q,p^{i}r\in\mathcal{Z}_{A}$, by Corollary~\ref{coro1}, we have
  \begin{align*}
    & \chi_{p^{i},p^n}(\sum_{j=0}^{q-1}(A_{jk}-A_{j0}))=0, \\
    & \chi_{p^{i},p^n}(\sum_{k=0}^{r-1}(A_{jk}-A_{0k}))=0.
  \end{align*}
  Then we can compute to get that
  \begin{align*}
  r\chi_{p^i,p^n}(\sum_{j=0}^{q-1}A_{jk})&=r\chi_{p^i,p^n}(\sum_{j=0}^{q-1}A_{j0})=\chi_{p^i,p^n}(\sum_{j=0}^{q-1}\sum_{k=0}^{r-1}A_{j0})=\chi_{p^i,p^n}(\sum_{j=0}^{q-1}\sum_{k=0}^{r-1}A_{jk})\\
  &=\chi_{p^i,p^n}(\sum_{j=0}^{q-1}\sum_{k=0}^{r-1}A_{0k})=q\chi_{p^i,p^n}(\sum_{k=0}^{r-1}A_{0k})=q\chi_{p^i,p^n}(\sum_{k=0}^{r-1}A_{jk}).
  \qedhere\end{align*}
\end{proof}

\section{Proof of Theorem~\ref{mainthm}}\label{section4}
Let $(A,B)$ be a nontrivial spectral pair in $\mathbb{Z}_{p^{n}qr}$. Assume that $A$ is not a tiling set, we will prove that there does not exist such spectral pair $(A,B)$.

Let $\mathbb{Z}_{p^{n}qr}=\langle a,b,c\rangle$, where $o(a)=p^{n}$, $o(b)=q$ and $o(c)=r$, and write $A=\sum_{j=0}^{q-1}\sum_{k=0}^{r-1}A_{jk}b^{j}c^{k}$ and $B=\sum_{j=0}^{q-1}\sum_{k=0}^{r-1}B_{jk}b^{j}c^{k}$, where $A_{jk},B_{jk}\in\mathbb{Z}_{\ge0}[\langle a\rangle]$. Let $e$ be the identity element of group $\mathbb{Z}_{p^{n}qr}$.

\begin{remark}
  \begin{enumerate}
    \item[(1)] If $a^{i_0},a^{i_0+up^{i_1}}\in A_{jk}$ for some $j\in[0,q-1],k\in[0,r-1]$ and $u\not\equiv0\pmod{p}$, then $p^{i_1}qr\in \mathcal{Z}_{B}$.
    \item[(2)] If $a^{i_0}\in A_{j_0k},a^{i_0+up^{i_1}}\in A_{j_1k}$ for some $j_0,j_1\in[0,q-1],k\in[0,r-1]$ with $j_0\ne j_1$, then $p^{i_1}r\in \mathcal{Z}_{B}$ when $u\not\equiv0\pmod{p}$, and $p^{n}r\in \mathcal{Z}_{B}$ when $u=0$.
    \item[(3)] If $a^{i_0}\in A_{jk_0},a^{i_0+up^{i_1}}\in A_{jk_1}$ for some $j\in[0,q-1],k_0,k_1\in[0,r-1]$ with $k_0\ne k_1$, then $p^{i_1}q\in \mathcal{Z}_{B}$ when $u\not\equiv0\pmod{p}$, and $p^{n}q\in \mathcal{Z}_{B}$ when $u=0$.
    \item[(4)] If $a^{i_0}\in A_{j_0k_0},a^{i_0+up^{i_1}}\in A_{j_1k_1}$ for some $j_0,j_1\in[0,q-1],k_0,k_1\in[0,r-1]$ with $j_0\ne j_1$ and $k_0\ne k_1$, then $p^{i_1}\in \mathcal{Z}_{B}$ when $u\not\equiv0\pmod{p}$, and $p^{n}\in \mathcal{Z}_{B}$ when $u=0$.
  \end{enumerate}
\end{remark}

Note that Fuglede's conjecture holds in $\mathbb{Z}_{p^{n}q}$ \cite{MK17}, $\mathbb{Z}_{pqr}$ \cite{S19} and $\mathbb{Z}_{p^{2}qr}$ \cite{Somlai21}, where $p,q,r$ are distinct primes. By Lemmas~\ref{lem3p1}, \ref{lem3p2}, \ref{lem3p3}, \ref{lem3p4} and \ref{lem3p5}, we also assume that
\begin{enumerate}
  \item [(1)] $e\in A$, $e\in B$;
  \item [(2)] $A$ generates group $\mathbb{Z}_{p^{n}qr}$;
  \item [(3)] $B$ generates group $\mathbb{Z}_{p^{n}qr}$;
  \item [(4)] $A$ is not a union of $\mathbb{Z}_{p}$- or $\mathbb{Z}_{q}$- or $\mathbb{Z}_{r}$-cosets exclusively.
\end{enumerate}

 Then $e\in A_{00}$ and $e\in B_{00}$. Denote
\begin{align*}
&I_{1}=\{i: i\in[0,n-1],p^{i}qr\in\mathcal{Z}_{A}\},\\
&I_{2}=\mathcal{Z}_{A}\cap \{p^{n}q,p^{n}r\},\\
&J_{1}=\{i: i\in[0,n-1],p^{i}qr\in\mathcal{Z}_{B}\},\\
&J_{2}=\mathcal{Z}_{B}\cap \{p^{n}q,p^{n}r\}.
\end{align*}
Then $0\le|I_1|,|J_1|\le n$ and $0\le|I_2|,|J_2|\le2$.
Now we first prove some lemmas.

\begin{lemma}\label{lem3p8}
\begin{enumerate}
  \item[(1)] If $q,r\in\mathcal{Z}_{A}$ and $qr\not\in\mathcal{Z}_{A}$, then $p^{n}q,p^{n}r\in\mathcal{Z}_{B}$.
  \item[(2)] If $q,r\in\mathcal{Z}_{B}$ and $qr\not\in\mathcal{Z}_{B}$, then $p^{n}q,p^{n}r\in\mathcal{Z}_{A}$.
\end{enumerate}
\end{lemma}
\begin{proof}
We will only prove the first statement, the proof of the second statement is similar. Note that $qr\not\in\mathcal{Z}_{A}$.
By Lemma~\ref{lemma3}, we have
\[r\chi_{1,p^n}(\sum_{j=0}^{q-1}A_{jk})=q\chi_{1,p^n}(\sum_{k=0}^{r-1}A_{jk})\ne0.\]
Let  $\sum_{j=0}^{q-1}A_{jk}=\sum_{i=0}^{p^{n}-1}x_{i}a^{i}$, and $\sum_{k=0}^{r-1}A_{jk}=\sum_{i=0}^{p^{n}-1}y_{i}a^{i}$, where $x_i,y_i\in\mathbb{Z}_{\ge0}$.
Then above inequations show that
\begin{align}
&\sum_{i=0}^{p^{n}-1}x_{i}\zeta_{p^n}^{i}\ne0,\label{eq1}\\
&\sum_{i=0}^{p^{n}-1}y_{i}\zeta_{p^n}^{i}\ne0,\label{eq2}\\
&\sum_{i=0}^{p^{n}-1}(rx_i-qy_i)\zeta_{p^n}^{i}=0.\label{eq3}
\end{align}
 By Lemma~\ref{lemma1}, Equation~(\ref{eq1}) implies that there exist $i_1,i_2$ with $i_1\equiv i_2\pmod{p^{n-1}}$ such that $x_{i_1}\ne x_{i_2}$. By Equation~(\ref{eq3}), we have $rx_{i_1}-qy_{i_1}=rx_{i_2}-qy_{i_2}$, which leads to $r(x_{i_1}-x_{i_2})=q(y_{i_1}-y_{i_2})$. Hence, we have $|x_{i_1}-x_{i_2}|\ge q$ and $|y_{i_1}-y_{i_2}|\ge r$. Therefore, $\max\{x_{i_1},x_{i_2}\}\ge q$ and $\max\{y_{i_1},y_{i_2}\}\ge r$. In other words, there exists $a^{i_0}\in\cup_{j=0}^{q-1}A_{jk}$ such that $a^{i_0}$ appears $q$ times in $\cup_{j=0}^{q-1}A_{jk}$. Hence $p^nr\in\mathcal{Z}_{B}$. Similarly, $p^{n}q\in\mathcal{Z}_{B}$.
\end{proof}

\begin{lemma}\label{lemma5}
  \begin{enumerate}
    \item [(1)] If $|J_{2}|\le1$, then $1\in\mathcal{Z}_{A}$.
    \item [(2)] If $|I_{2}|\le1$, then $1\in\mathcal{Z}_{B}$.
  \end{enumerate}
\end{lemma}
\begin{proof}
We will only prove the first statement, the proof of the second statement is similar.
Assume to the contrary, $1\not\in\mathcal{Z}_{A}$, by Lemma~\ref{lemma2}, $q,r\in\mathcal{Z}_{A}$. Then we have $qr\in\mathcal{Z}_{A}$ by  Lemma~\ref{lem3p8}. By Lemma~\ref{lemma3}, we have
\[\chi_{1,p^{n}}(\sum_{j=0}^{q-1}A_{jk})=\chi_{1,p^{n}}(\sum_{k=0}^{r-1}A_{jk})=0.\]
In other words, $\sum_{j=0}^{q-1}A_{jk}$ and $\sum_{k=0}^{r-1}A_{jk}$ are unions of some $p$-cycles. Since $1\notin\mathcal{Z}_{A}$, then there exist $j_1,k_1$ such that $\chi_{1,p^n}(A_{j_1k_1})\ne0$. Hence, there exists $a^{i_0}\in A_{j_1k_1}$ such that $a^{i_0+up^{n-1}}\notin A_{j_1k_1}$ for some $u\in[1,p-1]$. Moreover, $a^{i_0+up^{n-1}}\in A_{j_2k_1}$ and $a^{i_0+up^{n-1}}\in A_{j_1k_2}$ for some $j_2,k_2$ with $j_2\ne j_1$ and $k_2\ne k_1$.
This shows that $p^{n-1}q,p^{n-1}r,p^{n}\in\mathcal{Z}_{B}$. By Lemma~\ref{lemma3} and Corollary~\ref{coro1}, we have
\begin{align}
  & r\chi_{p^{n-1},p^n}(\sum_{j=0}^{q-1}B_{jk})=q\chi_{p^{n-1},p^n}(\sum_{k=0}^{r-1}B_{jk})\text{ for all }j\in[0,q-1],k\in[0,r-1],\label{eq4}\\
  & |B_{jk}|-|B_{j0}|-|B_{0k}|+|B_{00}|=0\text{ for all }j\in[0,q-1],k\in[0,r-1].\label{eq5}
\end{align}

{\bf{Claim:} $p^{n-1}\notin\mathcal{Z}_{B}$.}

Assume to the contrary, $p^{n-1}\in\mathcal{Z}_{B}$. If $p^{n-1}qr\in\mathcal{Z}_{B}$, by Lemma~\ref{lemma3}, we have $\chi_{p^{n-1},p^n}(B_{jk})=0$. Noting that $e\in B_{00}$, then
\[\{i\pmod{p}:\ a^{i}\in B_{00}\}=\{0,1,\dots,p-1\}.\]
 Since $1\not\in\mathcal{Z}_{A}$, then $B_{jk}=\emptyset$ for $j\in[1,q-1]$ and $k\in[1,r-1]$. If $B_{j0}\ne\emptyset$ for some $j\in[1,q-1]$, similarly as before,
 \begin{align*}
 &\{i\pmod{p}:\ a^{i}\in B_{j0}\}=\{0,1,\dots,p-1\},\\
 &B_{0k}=\emptyset\text{ for }k\in[1,r-1].
 \end{align*}
Thus $B=\sum_{j=0}^{q-1}B_{j0}b^{j}$, which contradicts to the fact that $B$ generates $\mathbb{Z}_{p^{n}qr}$. Similarly, if $B_{0k}\ne\emptyset$ for some $k\in[1,r-1]$, we can also get a contradiction.
Therefore, $p^{n-1}qr\not\in\mathcal{Z}_{B}$. By Lemma~\ref{lemma3}, we have
\[\chi_{p^{n-1},p^n}(B_{jk})=\chi_{p^{n-1},p^n}(B_{00})\ne0\text{ for all }j,k.\]
 Since $1\not\in\mathcal{Z}_{A}$, then $|\{i\pmod{p}:\ a^{i}\in B_{jk}\}|=1$ for all $j\in[0,q-1],k\in[0,r-1]$, and
\[\{i\pmod{p}:\ a^{i}\in B_{jk}\}=\{i\pmod{p}:\ a^{i}\in B_{00}\}.\]
Hence $|B_{jk}|=|B_{00}|$ for all $j\in[0,q-1],\ k\in[0,r-1]$. This shows that $p^{n}q,p^{n}r\in\mathcal{Z}_{B}$, which is a contradiction. This ends the proof of claim.

Now we divide our discussion into two cases.

{\bf{Case 1: $p$ is an odd prime.}}

Since $q,r,qr\in\mathcal{Z}_{A}$,  by Lemma~\ref{lemma3}, 
\[\chi_{1,p^n}(\sum_{j=0}^{q-1}A_{jk})=\chi_{1,p^n}(\sum_{k=0}^{r-1}A_{jk})=0\text{ for all }j,k.\]
In other words, $\sum_{j=0}^{q-1}A_{jk}$ and $\sum_{k=0}^{r-1}A_{jk}$ are unions of some $p$-cycles.
Note that $1\notin\mathcal{Z}_{A}$. There exist $j_1,k_1$ such that $\chi_{1,p^n}(A_{j_1k_1})\ne0$. Hence, there exists $a^{i_0}\in A_{j_1k_1}$ such that at least 2 of $a^{i_0+tp^{n-1}}$, $t=1,\dots,p-1$ do not belong to $A_{j_1k_1}$, say $a^{i_0+p^{n-1}}$ and $a^{i_0+2p^{n-1}}$ (if there are $p-1$ of $a^{i_0+tp^{n-1}}$, $t=0,\dots,p-1$ belong to $A_{j_1k_1}$ and the remaining one belong to $A_{j_1k_1^{'}}$, then change $A_{j_1k_1}$ to $A_{j_1k_1^{'}}$). Moreover,  $a^{i_0+p^{n-1}}\in A_{j_2k_1}$ and $a^{i_0+2p^{n-1}}\in A_{j_1k_2}$ for some $j_2,k_2$ with $j_2\ne j_1$ and $k_2\ne k_1$.
Therefore, $p^{n-1}\in\mathcal{Z}_{B}$, which is a contradiction.

{\bf{Case 2: $p=2$.}}

We divide our discussion into two subcases.

{\bf{Subcase 2.1: For all $j,k$, $|\{i\pmod{2}:\ a^{i}\in B_{jk}\}|\le1$.}}

{\bf{Claim: $B_{jk}=\emptyset$ for all $j\in[1,q-1],k\in[1,r-1]$.}}

Assume to the contrary, there exist $j_0\in[1,q-1]$, $k_0\in[1,r-1]$ such that $B_{j_0k_0}\ne\emptyset$.
 Note that $e\in B_{00}$ and $1\notin\mathcal{Z}_{A}$. We can get that
\begin{align*}
&\{i\pmod{2}:\ a^{i}\in B_{j_0k_0}\}=\{0\},\\
&1\notin\{i\pmod{2}:\ a^{i}\in \cup_{j\in[0,q-1],k\in[0,r-1]}B_{jk}\backslash(B_{j_00}\cup B_{0k_0})\}.
\end{align*}
Since $B$ generates $\mathbb{Z}_{p^{n}qr}$, then $1\in\{i\pmod{2}:\ a^{i}\in \cup_{j\in[0,q-1],k\in[0,r-1]}B_{jk}\}$. Hence $1\in\{i\pmod{2}:\ a^{i}\in B_{j_00}\cup B_{0k_0}\}$. If both $B_{j_00}$ and $B_{0k_0}$ are nonempty, then
 $\{i\pmod{2}:\ a^{i}\in B_{0k_0}\}=\{i\pmod{2}:\ a^{i}\in B_{j_00}\}=\{1\}$ and
\[B_{jk}=\emptyset\text{ for all }(j,k)\ne(0,0),(j_0,k_0),(j_0,0),(0,k_0).\]
For any $j_1\ne j_0$, $k_1\ne k_0$, by Equation~(\ref{eq5}), we have $|B_{j_1k_1}|-|B_{j_10}|-|B_{0k_1}|+|B_{00}|=0$. Then $|B_{00}|=0$, which is a contradiction. If only one of $B_{j_00}$ and $B_{0k_0}$ is nonempty, say $B_{0k_0}$, then $\{i\pmod{2}:\ a^{i}\in B_{0k_0}\}=\{1\}$, $B_{j_00}=\emptyset$ and
\[\{i\pmod{2}:\ a^{i}\in \cup_{j\in[1,q-1]}B_{jk_0}\}=\{i\pmod{2}:\ a^{i}\in \cup_{k\in[0,r-1]\backslash\{k_0\}}B_{0k}\}=\{0\}.\]
  By Equation~(\ref{eq5}), we have
\begin{align*}
&|B_{0k_1}|=|B_{0k_2}|\text{ for all }k_1,k_2\in[0,r-1]\backslash\{k_0\},\\
&|B_{j_1k_0}|=|B_{j_2k_0}|\text{ for all }j_1,j_2\in[1,q-1],\\
&|B_{0k_0}|=|B_{0k_1}|+|B_{j_1k_0}|\text{ for all }k_1\in[0,r-1]\backslash\{k_0\},j_1\in[1,q-1].
\end{align*}
Since $p^{n-1}q,p^{n-1}r\in\mathcal{Z}_{B}$, by Lemma~\ref{lemma3},
\begin{align*}
&\chi_{p^{n-1},p^n}(\sum_{j=0}^{q-1}(B_{jk_0}-B_{j0}))=0,\\
&\chi_{p^{n-1},p^n}(\sum_{k=0}^{r-1}(B_{j_0k}-B_{0k}))=0.
\end{align*}
From above equations, we can get
 \begin{align*}
&|B_{00}|=\sum_{j=1}^{q-1}|B_{jk_0}|-|B_{0k_0}|=(q-1)|B_{j_0k_0}|-|B_{0k_0}|,\\
&|B_{j_0k_0}|=\sum_{k\in[0,r-1]\backslash\{k_0\}}|B_{0k}|-|B_{0k_0}|=(r-1)|B_{00}|-|B_{0k_0}|,
\end{align*}
which contradicts to $|B_{0k_0}|=|B_{00}|+|B_{j_0k_0}|$. This ends the proof of claim.

 Since $B$ generates $\mathbb{Z}_{p^{n}qr}$, then $\cup_{j=1}^{q-1}B_{j0}\ne\emptyset$ and $\cup_{k=1}^{r-1} B_{0k}\ne\emptyset$. Note that $1\notin\mathcal{Z}_{A}$. We can get that
\[\{i\pmod{2}:\ a^{i}\in \cup_{j=1}^{q-1}B_{j0}\}=\{i\pmod{2}:\ a^{i}\in \cup_{k=1}^{r-1}B_{0k}\}=\{1\}.\]
 By Equation~(\ref{eq5}), we have
   \[|B_{00}|=|B_{0k}|+|B_{j0}|\text{ for }j\in[1,q-1],k\in[1,r-1],\] which leads to
 \begin{align*}
&|B_{0k_1}|=|B_{0k_2}|\text{ for }k_1,k_2\in[1,r-1],\\
&|B_{j_10}|=|B_{j_20}|\text{ for }j_1,j_2\in[1,q-1].
\end{align*}
Since $p^{n-1}q,p^{n-1}r\in\mathcal{Z}_{B}$, by Lemma~\ref{lemma3},
\begin{align*}
&\chi_{p^{n-1},p^n}(\sum_{j=0}^{q-1}(B_{jk}-B_{j0}))=0,\\
&\chi_{p^{n-1},p^n}(\sum_{k=0}^{r-1}(B_{jk}-B_{0k}))=0.
\end{align*}
In other words,
\begin{align}
&|B_{00}|-(q-1)|B_{10}|=|B_{00}|-\sum_{j=1}^{q-1}|B_{j0}|=-\sum_{j=0}^{q-1}|B_{j1}|=-|B_{01}|,\label{eq10}\\
&|B_{00}|-(r-1)|B_{01}|=|B_{00}|-\sum_{k=1}^{r-1}|B_{0k}|=-\sum_{k=0}^{r-1}|B_{1k}|=-|B_{10}|.\label{eq11}
\end{align}
Combining above two equations, we have $q|B_{10}|=r|B_{01}|$. Assume that $|B_{10}|=rm$ for some $m\in\mathbb{Z}_{>0}$, then $|B_{01}|=qm$ and $|B_{00}|=(q+r)m$. By Equation~(\ref{eq10}), we have $(q+r)m-(q-1)rm=-qm$, that is $(qr-2q-2r)m=0$, which contradicts to $2\nmid qr$.

{\bf{Subcase 2.2: There exist $j,k$ such that $\{i\pmod{2}:\ a^{i}\in B_{jk}\}=\{0,1\}$.}}

WLOG, assume that $\{i\pmod{2}:\ a^{i}\in B_{00}\}=\{0,1\}$. Since $1\not\in\mathcal{Z}_{A}$, then
\[B_{jk}=\emptyset\text{ for all }j\in[1,q-1],k\in[1,r-1].\]
 Since $B$ generates $\mathbb{Z}_{p^{n}qr}$, then $\cup_{j=1}^{q-1}B_{j0}\ne\emptyset$ and $\cup_{k=1}^{r-1} B_{0k}\ne\emptyset$. Note that $1\not\in\mathcal{Z}_{A}$. We can get that
 \begin{align*}
&|\{i\pmod{2}:\ a^{i}\in \cup_{j=1}^{q-1}B_{j0}\}|=|\{i\pmod{2}:\ a^{i}\in \cup_{k=1}^{r-1}B_{0k}\}|=1,\\
&\{i\pmod{2}:\ a^{i}\in \cup_{j=1}^{q-1}B_{j0}\}=\{i\pmod{2}:\ a^{i}\in \cup_{k=1}^{r-1}B_{0k}\}.
\end{align*}
   WLOG, assume that $\{i\pmod{2}:\ a^{i}\in \cup_{j=1}^{q-1}B_{j0}\}=\{i\pmod{2}:\ a^{i}\in \cup_{k=1}^{r-1}B_{0k}\}=\{0\}$. By Equation~(\ref{eq5}),
   \[|B_{00}|=|B_{0k}|+|B_{j0}|\text{ for }j\in[1,q-1],k\in[1,r-1],\] which leads to
 \begin{align*}
&|B_{0k_1}|=|B_{0k_2}|\text{ for }k_1,k_2\in[1,r-1],\\
&|B_{j_10}|=|B_{j_20}|\text{ for }j_1,j_2\in[1,q-1].
\end{align*}
Since $p^{n-1}q,p^{n-1}r\in\mathcal{Z}_{B}$, by Lemma~\ref{lemma3},
\begin{align*}
&\chi_{p^{n-1},p^n}(\sum_{j=0}^{q-1}(B_{jk}-B_{j0}))=0,\\
&\chi_{p^{n-1},p^n}(\sum_{k=0}^{r-1}(B_{jk}-B_{0k}))=0.
\end{align*}
Let $u=|\{b\in B_{00}: b\pmod{2}=0\}|$ and $v=|\{b\in B_{00}: b\pmod{2}=1\}|$, then we have
\begin{align}
&u+v=|B_{00}|=|B_{10}|+|B_{01}|,\label{eq8}\\
&(u-v)+(q-1)|B_{10}|=\chi_{p^{n-1},p^n}(B_{00})+\sum_{j=1}^{q-1}|B_{j0}|=\sum_{j=0}^{q-1}|B_{j1}|=|B_{01}|,\label{eq6}\\
&(u-v)+(r-1)|B_{01}|=\chi_{p^{n-1},p^n}(B_{00})+\sum_{k=1}^{r-1}|B_{0k}|=\sum_{k=0}^{r-1}|B_{1k}|=|B_{10}|.\label{eq7}
\end{align}
By Equations~(\ref{eq6}) and (\ref{eq7}), we have $r|B_{01}|=q|B_{10}|$. Assume that $|B_{10}|=rm$ for some $m\in\mathbb{Z}_{>0}$, then $|B_{01}|=qm$ and $|B_{00}|=(q+r)m$. By Equations~(\ref{eq8}) and (\ref{eq6}), we get
\begin{align*}
&u-v=(q+r-qr)m,\\
&u+v=(q+r)m.
\end{align*}
Combining above two equations, we obtain $2u=(2q+2r-qr)m\ge0$. Since $q,r$ are distinct odd primes, then $(q,r)=(3,5)$ or $(5,3)$. WLOG, assume that $q=3$ and $r=5$. Then we can get that $|B_{00}|=8m$, $|B_{j0}|=5m$, $|B_{0k}|=3m$ and $|A|=|B|=30m$. By the pigeonhole principle, we have $|I_{1}|\ge\log_{2}(8m)$, that is $2^{|I_1|}\ge 8m$. On the other hand, by the definition of $I_1$, we have $2^{|I_1|}\mid |A|$, and then $2^{|I_1|}\mid 2m$, which is a contradiction.
\end{proof}

\begin{lemma}\label{lemma4}
  $|J_2|\ge1.$
\end{lemma}
\begin{proof}
If $|J_2|=0$, then $1\in\mathcal{Z}_{A}$ by Lemma~\ref{lemma5}.
By Corollary~\ref{coro1}, we have
\[\chi_{1,p^n}(A_{jk}-A_{j0}-A_{0k}+A_{00})=0\text{ for all }j,k.\]
 Since $p^{n}q,p^{n}r\not\in\mathcal{Z}_{B}$, then
 \[(A_{jk}\cup A_{00})\cap(A_{j0}\cup A_{0k})=\emptyset.\]
  Thus $\chi_{1,p^n}(A_{jk}+A_{00})=0$. If there exist $j,k$ such that $\chi_{1,p^n}(A_{jk})\ne0$, then $\chi_{1,p^n}(A_{00})\ne0$. Hence, there exists $a^{i_0}\in A_{00}$, but $a^{i_0+tp^{n-1}}\notin A_{00}$ for some $t\in[1,p-1]$, then $a^{i_0+tp^{n-1}}\in A_{jk}$. Similarly, we can get that $a^{i_0+tp^{n-1}}\in A_{jk'}$ for some $k'\ne k$.
Hence, $p^{n}q\in\mathcal{Z}_{B}$, which is a contradiction. Therefore, $\chi_{1}(A_{jk})=0$ for all $j\in[0,q-1],k\in[0,r-1]$. This shows that $A$ is a union of $\mathbb{Z}_{p}$-cosets, which is also a contradiction.
\end{proof}

\begin{lemma}\label{lemma6}
  \begin{enumerate}
    \item[(1)] $p^{n}r\in\mathcal{Z}_{A}$ or $p^{n}r\in\mathcal{Z}_{B}$;
    \item[(2)] $p^{n}q\in\mathcal{Z}_{A}$ or $p^{n}q\in\mathcal{Z}_{B}$.
  \end{enumerate}
\end{lemma}
\begin{proof}
We will only prove the first statement, the proof of the second statement is similar.
Assume to the contrary, $p^{n}r\notin\mathcal{Z}_{A}$ and $p^{n}r\notin\mathcal{Z}_{B}$. By Lemmas~\ref{lemma5} and \ref{lemma4}, $1\in\mathcal{Z}_{A}$ and $1,p^{n}q\in\mathcal{Z}_{B}$. Then $r\mid|A|$, we may assume that $|A|=p^{t}rm$, where $\gcd(p,m)=1$.

If $r\in\mathcal{Z}_{A}$, by Lemma~\ref{lemma3}, $\chi_{1,p^n}(A_{jk}-A_{0k})=0$. Since $p^{n}r\notin\mathcal{Z}_{B}$, then $A_{jk}\cap A_{0k}=\emptyset$. Hence, $\chi_{1,p^n}(A_{jk})=0$. This shows that $A$ is a union of $\mathbb{Z}_{p}$-cosets, which is a contradiction. Therefore $r\not\in\mathcal{Z}_{A}$.

{\bf{Claim: $p^{n-1}qr\in\mathcal{Z}_{B}$.}}

Assume to the contrary, $p^{n-1}qr\not\in\mathcal{Z}_{B}$. Since $1\in\mathcal{Z}_{A}$, by Corollary~\ref{coro1},
\[\chi_{1,p^n}(A_{jk}-A_{j0}-A_{0k}+A_{00})=0.\]
 Then for any $a^{i_0}\in A_{00}$, we have $a^{i_0+up^{n-1}}\notin A_{00}$ for $u\in[1,p-1]$. Note that $p^{n}r\notin\mathcal{Z}_{B}$. We can get that $A_{j0}\cap A_{00}=\emptyset$, and then $a^{i_0}\notin A_{j0}$. If $a^{i_0}\notin A_{0k}$, then $a^{i_0+up^{n-1}}\in A_{jk}$.
 Considering $\chi_{1,p^n}(A_{j'k}-A_{j'0}-A_{0k}+A_{00})=0$ for some $j'\ne j$, similarly as before, we can get $a^{i_0+up^{n-1}}\in A_{j'k}$. This shows that $p^{n}r\in\mathcal{Z}_{B}$, which is a contradiction. Hence, $a^{i_0}\in A_{0k}$, and then
 $A_{0k}-A_{00}=0$. Therefore, $A_{jk}=A_{j0}$ for all $j,k$, and then $A=\sum_{j=0}^{q-1}A_{j0}b^{j}\sum_{k=0}^{r-1}c^{k}$, which contradicts to the fact that $A$ is not a union of $\mathbb{Z}_{r}$-cosets. This ends the proof of claim.

{\bf{Claim: $qr\in\mathcal{Z}_{A}$.}}

Assume to the contrary, $qr\not\in\mathcal{Z}_{A}$. Since $p^{n-1}qr,p^{n}q\in\mathcal{Z}_{B}$, by Lemma~\ref{lemma3}, we have
\begin{align*}
&\chi_{p^{n-1},p^n}(\sum_{j=0}^{q-1}\sum_{k=0}^{r-1}B_{jk})=0,\\
&\sum_{j=0}^{q-1}|B_{jk}|=p^{t}m.
\end{align*}
Note that $r,qr\not\in\mathcal{Z}_{A}$. We have $\{i\pmod{p}:\ a^{i}\in \cup_{j=0}^{q-1}B_{jk}\}=\{i_{k}\}$ for some $i_{k}\in[0,p-1]$. Then we can compute to get that
\begin{align*}
0&=\chi_{p^{n-1},p^n}(\sum_{j=0}^{q-1}\sum_{k=0}^{r-1}B_{jk})\\
 &=\sum_{k=0}^{r-1}\chi_{p^{n-1},p^n}(\sum_{j=0}^{q-1}B_{jk})\\
 &=\sum_{k=0}^{r-1}\sum_{j=0}^{q-1}|B_{jk}|e^{\frac{2\pi i\cdot i_{k}}{p}}\\
 &=p^{t}m\sum_{k=0}^{r-1}e^{\frac{2\pi i\cdot i_{k}}{p}},
\end{align*}
which contradicts to $p\nmid r$. This ends the proof of claim.

Since $r\notin\mathcal{Z}_{A}$, WLOG, the nonempty sets of $B_{jk}$ are as follows (after permuting the rows and columns of $(B_{jk})_{j\in[0,q-1],k\in[0,r-1]}$)
 {\footnotesize{
\[
   \begin{array}{cccccccccccccc}
     B_{00}&\cdots&B_{0,s_{0}-1} &   & & &  & B_{0,s_{u}}  &  \cdots & B_{0,s_{u+1}-1}  & \cdots  & B_{0,s_{u+p-1}}  &   & B_{0,s_{u+p}-1}  \\
       &   & & \ddots & & &  & \vdots  & \vdots  & \vdots  & \cdots  & \vdots  & \vdots  & \vdots  \\
       &   & &   & B_{u,s_{u-1}}&\cdots&B_{u,s_{u}-1} & \cdots  & \cdots  & \cdots  & \cdots  &  \cdots &  \cdots & \cdots  \\
       &   &&   &  & & & \vdots & \vdots & \vdots &  \vdots & \vdots  & \vdots  &  \vdots \\
       &   & &   & & &  & B_{q-1,s_{u}} & \cdots & B_{q-1,s_{u+1}-1}  & \cdots  & B_{q-1,s_{u+p-1}}  &   & B_{q-1,s_{u+p}-1}
   \end{array},
 \]}}
where $0=:s_{-1}\le s_{0}\le s_{1}\le\cdots\le s_{u+p}:=r$, $|\{i\pmod{p}:\ a^{i}\in B_{jk}\}|\ge2$ for $j\in[0,u]$, $k\in[s_{j-1},s_{j}-1]$, and $\{i\pmod{p}:\ a^{i}\in \cup_{j=0}^{q-1}B_{jk}\}=i$ for $i\in[0,p-1]$, $k\in[s_{u+i},s_{u+i+1}-1]$. Since $p^{n}q\in\mathcal{Z}_{B}$, we have
  \begin{align*}
  &|B_{jk}|=p^{t}m \text{ for }j\in[0,u],k\in[s_{j-1},s_{j}-1],\\
  &\sum_{j=0}^{q-1}|B_{jk}|=p^{t}m\text{ for }k\in[s_{u},s_{u+p}-1].
  \end{align*}

  {\bf{Case 1}: $0< s_{u}<r$.}

For this case, $|I_1|\le t$. By the pigeonhole principle, we have $|I_1|\ge \log_{p}(p^{t}m)$. Then $m=1$.
Note that $1,qr\in\mathcal{Z}_{A}$. If $q\in\mathcal{Z}_{A}$, by Lemma~\ref{lemma3},
\begin{align*}
&\chi_{1,p^n}(A_{jk}-A_{j0})=0,\\
&\chi_{1,p^n}(\sum_{j=0}^{q-1}A_{jk})=0.
\end{align*}
Since $r\notin\mathcal{Z}_{A}$, then there exists $j,k$ such that $\chi_{1,p^n}(A_{jk})\ne0$. Hence, there exists $a^{i_0}\in A_{jk}$ but $a^{i_0+up^{n-1}}\notin A_{jk}$ for some $u\in[1,p-1]$. Moreover, $a^{i_0}\in A_{j0}$ and $a^{i_0+up^{n-1}}\in A_{j'k}$ for some $j'\ne j$.
This shows that $p^{n-1},p^{n-1}r\in\mathcal{Z}_{B}$. By Lemma~\ref{lemma3}, we have $\chi_{p^{n-1},p^n}(B_{jk}-B_{0k})=0$. Then we deduce that $|B_{jk}|=|B_{0k}|$ for $j\in[0,q-1],k\in[s_{u},s_{u+p}-1]$, which contradicts to $\sum_{j=0}^{q-1}|B_{jk}|=p^{t}$ for $k\in[s_{u},s_{u+p}-1]$. Hence $q\not\in\mathcal{Z}_{A}$.
Therefore, the nonempty set of $B_{jk}$ are as follows
{\footnotesize{\[
   \begin{array}{cccccccccccccc}
     B_{00} &   &  &   &   &   &   &   &   &   \\
       &    \ddots    &  &   &   &   &   &   &   &   \\
       &      & B_{u,u}&   &   &   &   &   &   &   \\
       &      &  & B_{u+1,s_{u}} & \cdots & B_{u+1,s_{u+1}-1} &   &   &   &   \\
       &     &  & \vdots & \ddots & \vdots &   &   &   &   \\
       &      & & B_{u+j_{1},s_{u}} & \cdots & B_{u+j_{1},s_{u+1}-1} &   &   &    &   \\
       &      &  &   &   &   & \ddots &   &   &   \\
       &      & &   &   &   &   & B_{u+j_{p-1}+1,s_{u+p-1}} & \cdots & B_{u+j_{p-1}+1,s_{u+p}-1} \\
       &      &  &   &   &   &   & \vdots & \ddots & \vdots \\
      &      &  &   &   &   &   & B_{u+j_{p},s_{u+p-1}} & \cdots & B_{u+j_{p},s_{u+p}-1}
   \end{array}.
 \]}}
 Since $\cup_{i=0}^{q-1}\{i\pmod{p}:\ a^{i}\in B_{jk}\}=i$ for $i\in[0,p-1]$, $k\in[s_{u+i},s_{u+i+1}-1]$, then $p^{n-1},p^{n-1}q,p^{n-1}r\notin\mathcal{Z}_{B}$. Hence, for any $j_1\in[0,q-1],\ k_1\in[0,r-1]$ and $a^{i_0}\in A_{j_1k_1}$, we have $a^{i_0+up^{n-1}}\notin A_{jk}$ for all $u\in[1,p-1]$, $j\ne j_1$ or $k\ne k_1$.
Note that $qr\in\mathcal{Z}_{A}$. By Corollary~\ref{coro1}, $\chi_{1,p^n}(\sum_{j=0}^{q-1}\sum_{k=0}^{r-1}A_{jk})=0$. Then
 we have $\chi_{1,p^n}(A_{jk})=0$ for all $j\in[0,q-1],k\in[0,r-1]$, which contradicts to $q,r\notin\mathcal{Z}_{A}$.

{\bf{Case 2}: $s_{u}=0$.}

For this case, the nonempty sets of $B_{jk}$ are as follows
\[
   \begin{array}{ccccccc}
      B_{0,s_{u}}  &  \cdots & B_{0,s_{u+1}-1}  & \cdots  & B_{0,s_{u+p-1}}  &   & B_{0,s_{u+p}-1}  \\
       \vdots  & \vdots  & \vdots  & \cdots  & \vdots  & \vdots  & \vdots  \\
       \cdots  & \cdots  & \cdots  & \cdots  &  \cdots &  \cdots & \cdots  \\
        \vdots & \vdots & \vdots &  \vdots & \vdots  & \vdots  &  \vdots \\
      B_{q-1,s_{u}} & \cdots & B_{q-1,s_{u+1}-1}  & \cdots  & B_{q-1,s_{u+p-1}}  &   & B_{q-1,s_{u+p}-1}  \\
   \end{array},
 \]
where $0=: s_{u}\le\cdots\le s_{u+p}:=r$, and $\{i\pmod{p}:\ a^{i}\in \cup_{j=0}^{q-1}B_{jk}\}=i$ for $i\in[0,p-1]$, $k\in[s_{u+i},s_{u+i+1}-1]$.

Note that $1,qr\in\mathcal{Z}_{A}$. If $q\in\mathcal{Z}_{A}$, by Lemma~\ref{lemma3},
\begin{align*}
&\chi_{1,p^n}(A_{jk}-A_{j0})=0,\\
&\chi_{1,p^n}(\sum_{j=0}^{q-1}A_{jk})=0.
\end{align*}
Since $r\notin\mathcal{Z}_{A}$, then there exist $j,k$ such that $\chi_{1,p^n}(A_{jk})\ne0$. Hence, there exists $a^{i_0}\in A_{jk}$ but $a^{i_0+up^{n-1}}\notin A_{jk}$ for some $u\in[1,p-1]$. Moreover, $a^{i_0}\in A_{j0}$ and $a^{i_0+up^{n-1}}\in A_{j'k}$.
This shows that $p^{n-1},p^{n-1}r\in\mathcal{Z}_{B}$. By Lemma~\ref{lemma3}, we have $\chi_{p^{n-1},p^n}(B_{jk}-B_{0k})=0$. Thus $|B_{jk}|=|B_{0k}|$ for $k\in[s_{u},s_{u+p}-1]$. Therefore, $|B|=\sum_{j,k}|B_{jk}|=p^{t}qrm'$, and $|B_{jk}|=p^{t}m'$ for all $j,k$, which contradicts to $p^{n}r\notin\mathcal{Z}_{B}$. Hence, $q\not\in\mathcal{Z}_{A}$, and the nonempty set of $B_{jk}$ are as follows
\[
   \begin{array}{cccccccccccccc}
       B_{0,s_{u}} & \cdots & B_{0,s_{u+1}-1} &   &   &   &   \\
        \vdots & \ddots & \vdots &   &   &   &   \\
        B_{j_{1},s_{u}} & \cdots & B_{j_{1},s_{u+1}-1} &   &   &    &   \\
        &   &   & \ddots &   &   &   \\
        &   &   &   & B_{j_{p-1}+1,s_{u+p-1}} & \cdots & B_{j_{p-1}+1,s_{u+p}-1} \\
        &   &   &   & \vdots & \ddots & \vdots \\
   &   &   &   & B_{j_{p},s_{u+p-1}} & \cdots & B_{j_{p},s_{u+p}-1} \\
   \end{array}.
 \]
Since $\{i\pmod{p}:\ a^{i}\in \cup_{i=0}^{q-1}B_{jk}\}=i$ for $i\in[0,p-1]$, $k\in[s_{u+i},s_{u+i+1}-1]$, then $p^{n-1}q,p^{n-1}r\notin\mathcal{Z}_{B}$.
If $p^{n-1}\in\mathcal{Z}_{B}$, by Corollary~\ref{coro1}, we have
\[\chi_{p^{n-1},p^n}(B_{j_{p-1}+1,s_{u+p-1}}-B_{j_{p-1}+1,0}-B_{0,s_{u+p-1}}+B_{00})=0.\]
That is $\chi_{p^{n-1},p^n}(B_{j_{p-1}+1,s_{u+p-1}}+B_{00})=0$. Note that $\{i\pmod{p}:\ a^{i}\in B_{00}\}=0$ and  $\{i\pmod{p}:\ a^{i}\in B_{j_{p-1}+1,s_{u+p-1}}\}=p-1$. We deduce that $p=2$ and $|B_{j_{p-1}+1,s_{u+p-1}}|=|B_{00}|$. A similar discussion as above, we can get that
 all nonempty $B_{jk}$ have the same size. Hence $j_p<q-1$. By Corollary~\ref{coro1}, we have
\[\chi_{p^{n-1},p^n}(B_{q-1,r-1}-B_{q-1,0}-B_{0,r-1}+B_{00})=0.\]
This shows that $\chi_{p^{n-1},p^n}(B_{00})=0$.
This contradicts to $\{i\pmod{p}:\ a^{i}\in B_{00}\}=0$. 
Hence $p^{n-1}\notin\mathcal{Z}_{B}$. Note that $p^{n-1}q,p^{n-1}r\notin\mathcal{Z}_{B}$. Then for any $j_1\in[0,q-1],\ k_1\in[0,r-1]$, $a^{i_0}\in A_{j_1k_1}$, and $u\in[1,p-1]$, we have $a^{i_0+up^{n-1}}\notin A_{jk}$ for all $j\ne j_1$, $k\ne k_1$.
Since $qr\in\mathcal{Z}_{A}$, by Corollary~\ref{coro1},
\[\chi_{1,p^n}(\sum_{j=0}^{q-1}\sum_{k=0}^{r-1}A_{jk})=0.\]
Then we have $\chi_{1,p^n}(A_{jk})=0$ for all $j,k$, which contradicts to $q,r\notin\mathcal{Z}_{A}$.

{\bf{Case 3}: $s_{u}=r$.}

For this case, the nonempty set of $B_{jk}$ are as follows
\[
   \begin{array}{ccccccc}
     B_{00}&\cdots&B_{0,s_{0}-1} &   & & &   \\
       &   & & \ddots & & &   \\
       &   & &   & B_{u,s_{u-1}}&\cdots&B_{u,s_{u}-1}
   \end{array},
 \]
 where $s_u=r$ and $u\le q-1$.
Then $|B_{jk}|=p^{t}m$ for $j\in[0,u]$ and $k\in[s_{j-1},s_{j}-1]$. By the pigeonhole principle, we have $m=1$, $|I_{1}|=t$, and $|B_{jk}|=p^{t}$ for $j\in[0,u]$, $k\in[s_{j-1},s_{j}-1]$. Hence $|A|=|B|=p^{t}r$.

{\bf{Claim: $n-1\in I_1$.}}

Note that $1\in\mathcal{Z}_{B}$.  If $u<q-1$,
  by Corollary~\ref{coro1},
   \[\chi_{1,p^n}(B_{u,s_{u-1}}-B_{u,0}-B_{q-1,s_{u-1}}+B_{q-1,0})=0.\]
  Then we get $\chi_{1,p^n}(B_{u,s_{u-1}})=0$. Similarly, we can get that $\chi_{1,p^n}(B_{jk})=0$ for all $j\in[0,q-1]$, $k\in[0,r-1]$. Hence $n-1\in I_1$.

  If $u=q-1$ and $q\ge3$,  by Corollary~\ref{coro1},
 \[\chi_{1,p^n}(B_{u,s_{u-1}}-B_{u,0}-B_{u-1,s_{u-1}}+B_{u-1,0})=0.\]
 Then we get $\chi_{1,p^n}(B_{u,s_{u-1}})=0$. Similarly, we can get that $\chi_{1,p^n}(B_{jk})=0$ for all $j\in[0,q-1]$, $k\in[0,r-1]$.  Hence $n-1\in I_1$.

If $u=q-1$ and $q=2$,  by Corollary~\ref{coro1},
   \[\chi_{1,p^n}(B_{1,s_{0}}-B_{1,0}-B_{0,s_{0}}+B_{0,0})=0.\]
   That is $\chi_{1,p^n}(B_{1,s_{0}}+B_{0,0})=0$.
Since $p\ne2$, then there are at least two of $e,a^{p^{n-1}},a^{2p^{n-1}}$ belong to $B_{00}$ or $B_{1,s_{0}}$. Hence $n-1\in I_1$. This ends the proof of claim.

By Lemma~\ref{lemma7},
we have
\[B_{jk}=\{a^{c_{jk}+\sum_{i\in I_{1}}a_ip^i}: a_i\in[0,p-1]\}\]
 for $j\in[0,u]$ and $k\in[s_{j-1},s_{j}-1]$. Then we can compute to get that $\chi_{p^{n-1-i},p^n}(B_{jk})=0$ for any $i\in I_1$, $j\in[0,q-1]$, $k\in[0,r-1]$. Hence $J_1=\{n-1-i: i\in I_1\}.$

 {\bf{Claim : $p^{n}\notin\mathcal{Z}_{B}$.}}

If $p^{n}\in\mathcal{Z}_{B}$, then by Corollary~\ref{coro1},
 \[|B_{u,s_{u-1}}|-|B_{0,s_{u-1}}|-|B_{u,0}|+|B_{00}|=0.\]
We have $|B_{u,s_{u-1}}|+|B_{00}|=0$, which is a contradiction. Hence $p^{n}\notin\mathcal{Z}_{B}$. This ends the proof of claim.

 {\bf{Claim : $p^{i}\in\mathcal{Z}_B$ if and only if $p^{i}qr\in\mathcal{Z}_B$.}}

From above discussion, we have seen that
 if $p^{i}qr\in\mathcal{Z}_B$, then $\chi_{p^{i},p^n}(B_{jk})=0$ for any $j\in[0,q-1]$, $k\in[0,r-1]$, and so $p^{i}\in\mathcal{Z}_B$. Now we assume that $p^{i}\in\mathcal{Z}_B$.

 If $u<q-1$,
  by Corollary~\ref{coro1},
   \[\chi_{p^i,p^n}(B_{u,s_{u-1}}-B_{u,0}-B_{q-1,s_{u-1}}+B_{q-1,0})=0.\]
  Then we get $\chi_{p^i,p^n}(B_{u,s_{u-1}})=0$. Similarly, we can get that $\chi_{p^i,p^n}(B_{jk})=0$ for all $j\in[0,q-1]$, $k\in[0,r-1]$. Hence $p^{i}qr\in\mathcal{Z}_B$.

  If $u=q-1$ and $q\ge3$,  by Corollary~\ref{coro1},
 \[\chi_{p^i,p^n}(B_{u,s_{u-1}}-B_{u,0}-B_{u-1,s_{u-1}}+B_{u-1,0})=0.\]
 Then we get $\chi_{p^i,p^n}(B_{u,s_{u-1}})=0$. Similarly, we can get that $\chi_{p^i,p^n}(B_{jk})=0$ for all $j\in[0,q-1]$, $k\in[0,r-1]$. Hence $p^{i}qr\in\mathcal{Z}_B$.

   If $u=q-1$ and $q=2$,  by Corollary~\ref{coro1},
 \[\chi_{p^i,p^n}(B_{1,s_{0}}-B_{1,0}-B_{0,s_{0}}+B_{00})=0.\]
 That is
 \begin{align*}
 0=\chi_{p^i,p^n}(B_{1,s_{0}}+B_{00})&=\chi_{p^i,p^n}(a^{c_{1,s_{0}}})\chi_{p^i,p^n}(\sum_{i\in I_{1}}a_ip^i)+\chi_{p^i,p^n}(a^{c_{00}})\chi_{p^i,p^n}(\sum_{i\in I_{1}}a_ip^i)\\
 &=\chi_{p^i,p^n}(a^{c_{1,s_{0}}}+a^{c_{00}})\chi_{p^i,p^n}(\sum_{i\in I_{1}}a_ip^i).
 \end{align*}
 Since $p\ne2$, then $\chi_{p^i,p^n}(a^{c_{1,s_{0}}}+a^{c_{00}})\ne0$, and so $\chi_{p^i,p^n}(\sum_{i\in I_{1}}a_ip^i)=0$. Hence  $p^{i}qr\in\mathcal{Z}_B$.
 This ends the proof of claim.

 From above two claims, we have
 \begin{align*}
 &A_{j_1k_1}\cap A_{j_2k_2}=\emptyset\text{ for any }j_1\ne j_2, k_1\ne k_2,\\
 &(A_{j_1k_1}+A_{j_2k_2})(A_{j_1k_1}+A_{j_2k_2})^{(-1)}\subset\{a^i: i\in J_1\}.
 \end{align*}
 Then we get
 \[(\sum_{j=0}^{q-1}A_{j,k+j})(\sum_{j=0}^{q-1}A_{j,k+j})^{(-1)}\subset\{a^i: i\in J_1\}\text{ for all }k=0,1,\dots,r-1,\]
 where the second subscript of $A_{i,i+j}$ is modulo $r$. Since $|A|=p^{t}r$, by the pigeonhole principle, we have $|\sum_{j=0}^{q-1}A_{j,k+j}|=p^{t}$. By Lemma~\ref{lemma7},
 \[\sum_{j=0}^{q-1}A_{j,k+j}=\{a^{d_{k}+\sum_{i\in J_{1}}a_ip^i}: a_i\in[0,p-1]\}.\]
 A similar discussion as above, we can get
  \[A_{0,k-1}+\sum_{j=1}^{q-1}A_{j,k+j}=\{a^{d_{k}+\sum_{i\in J_{1}}a_ip^i}: a_i\in[0,p-1]\}.\]
  This shows that $A_{0k}=A_{00}$ for all $k$. Similarly as above, we can show that $A_{jk}=A_{j0}$ for all $j,k$. Hence $A=\sum_{j=0}^{q-1}A_{j0}b^j\sum_{k=0}^{r-1}c^k$, which contradicts to $A$ is not a union of $\mathbb{Z}_{r}$-cycles.
\end{proof}
By Lemma~\ref{lemma6}, we have the following corollary.
\begin{corollary}
  $|I_2|+|J_2|\ge2$.
\end{corollary}

 Now we divide our discussion into 2 cases according to the size of $I_2,J_2$.

\subsection{$|I_{2}|=2$ or $|J_2|=2$}\label{subsec1}
Assume that $|I_1|=t$, then we have $p^{t}qr\mid|A|$. For $y\in[0,q-1],z\in[0,r-1]$, denote
\[B_{yz}=\{x: a^{x}b^yc^z\in B\}.\]
 If $|A|=|B|>p^{t}qr$, then there exist $y,z$ such that $|B_{yz}|>p^{t}$. By the pigeonhole principle, we have $|I_{1}|\ge t+1$, which is a contradiction.

 Now we assume that $|A|=p^{t}qr$, then $|J_{1}|\le t$. For $y\in[0,q-1],z\in[0,r-1]$, denote
 \[A_{yz}=\{x: a^{x}b^yc^z\in A\}.\]
By the pigeonhole principle again, we have $|A_{yz}|=p^{t}$ for any $y\in[0,q-1],z\in[0,r-1]$. Then $|J_{1}|=t$. Denote
\[T:=\{a^{\sum_{i\in[0,n-1]\backslash J_{1}}x_{i}p^{i}}: x_{i}\in[0,p-1]\}.\]
If $(AA^{(-1)})\cap(TT^{(-1)})\ne\{e\}$, then there exists $i\in[0,n-1]\backslash J_{1}$, such that $p^{i}qr\in\mathcal{Z}_{B}$, which is a contradiction. Hence $(AA^{(-1)})\cap(TT^{(-1)})=\{e\}$. By Lemma~\ref{lem3p2}, $(A,T)$ forms a tiling pair in $\mathbb{Z}_{p^{n}qr}$, which contradicts to $A$ is not a tiling set.

\subsection{$|I_{2}|=|J_{2}|=1$}
By Lemma~\ref{lemma5},
$1\in\mathcal{Z}_{A}$ and $1\in\mathcal{Z}_{B}$.
By Lemma~\ref{lemma6}, WLOG, we assume that $p^{n}q\in\mathcal{Z}_{A}$, $p^{n}r\not\in\mathcal{Z}_{A}$, $p^{n}r\in\mathcal{Z}_{B}$ and $p^{n}q\not\in\mathcal{Z}_{B}$.
Then $qr\mid|A|$. Assume that $|A|=p^{t}qrm$. For $y\in[0,q-1],z\in[0,r-1]$, denote
 \[A_{yz}=\{x: a^xb^yc^z\in A\}.\]
A similar discussion as Subsection~\ref{subsec1}, we can get that $m=1$, $|A|=p^tqr$ and $|A_{yz}|=p^t$ for $y\in[0,q-1],z\in[0,r-1]$. This shows that $p^{n}q,p^nr\in\mathcal{Z}_{A}$, which is a contradiction.

\end{document}